\documentclass[12pt]{article}

\usepackage{amsfonts, amssymb, amsmath, epsf}
\usepackage{amsthm}
\usepackage{dsfont}

\DeclareMathOperator{\dist}{dist}

\theoremstyle{plain}
\newtheorem{theorem}{Theorem}
\newtheorem{lemma}{Lemma}
\newtheorem{proposition}{Proposition}

\theoremstyle{definition}
\newtheorem{definition}{Definition}

\newtheorem{remark}{Remark}
\newtheorem{question}{Question}
\newtheorem{conjecture}{Conjecture}

\title{Shadowing for actions of some finitely generated groups.}
\author{Alexey V. Osipov
\footnote{Chebyshev Laboratory, Saint-Petersburg State University,
14th line of Vasilievsky island, 29B, Saint-Petersburg, 199178, Russia;
Centro di ricerca matematica Ennio de Giorgi, Scuola Normale Superiore,
Piazza dei Cavalieri 3,
Pisa, 56100, Italy;
osipovav28@googlemail.com}
\and
Sergey B. Tikhomirov \footnote{Free University Berlin, Arnimallee 3, Berlin, 14195, Germany; Chebyshev Laboratory, Saint-Petersburg State University,
14th line of Vasilievsky island, 29B, Saint-Petersburg, 199178, Russia; sergey.tikhomirov@gmail.com}
}
\begin{document}

\newcommand{\ep}{\varepsilon}
\newcommand{\lam}{\lambda}
\newcommand{\sref}[1]{(\ref{#1})}
\newcommand{\ZZ}{\mathds{Z}}
\newcommand{\RR}{\mathds{R}}

\maketitle

\begin{abstract}
We introduce a notion of shadowing property for actions of finitely generated groups and study its basic properties. We formulate and prove a shadowing lemma for actions of nilpotent groups. We construct an example of a faithful linear action of a solvable Baumslag-Solitar group and show that the shadowing property depends on quantitative characteristics of hyperbolicity. Finally we show that any linear action of a non-abelian free group does not have the shadowing property.
\end{abstract}

\textbf{keywords:} Shadowing, expansivity, group action, nilpotent group, solvable group, free group.

\section{Introduction. }

Theory of shadowing is now a sufficiently well-developed branch of theory of dynamical systems (see monographs \cite{PilBook, PalmBook} and a review of modern results~\cite{Pil2}). A dynamical system has a shadowing property if any sufficiently precise approximate trajectory (pseudotrajectory) is close to some exact trajectory. Shadowing theory plays an important role in theory of structural stability. The shadowing lemma \cite{Anosov, Bowen} is one of key results in theory of shadowing. It says that a dynamical system has the shadowing property in a small neighborhood of a hyperbolic set.

In parallel with a classical theory of dynamical systems (which studies actions of $\mathbb{Z}$ and $\mathbb{R}$), global qualitative properties of actions of more complicated groups were studied
(see the book \cite{Katok2} and the review \cite{Fisher}).
The paper \cite{PilTikh} introduced the shadowing property for actions of abelian groups $\mathbb{Z}^n \times \mathbb{R}^m$ for nonnegative integer $n$ and $m$.

In the present paper we introduce and study the shadowing property for actions of finitely generated, not necessarily, abelian groups.

For the case of finitely generated nilpotent groups we prove that an action of the whole group has shadowing (and expansivity) if the action of at least one element has shadowing and expansivity (Theorem~\ref{thmVirtNilp}).
This result can be viewed as a shadowing lemma for actions of nilpotent groups, since it implies that if an action of one element is hyperbolic, then the group action has the shadowing property. Note that in some cases an action of a  group is called hyperbolic if there exists an element which action is hyperbolic (see \cite{Barbot, Katok2, Hurder}). 

We show that our result cannot be directly generalized to the case of solvable groups. We consider a particular linear action of a solvable Baumslag-Solitar group (Theorem~\ref{ThBS}) and demonstrate that the shadowing property has a more complicated nature, in particular, it depends on quantitative characteristics of hyperbolicity of the action.

We also consider actions of ''big groups'' (free groups, groups with infinitely many ends). In particular we show that there is no linear action of a non-abelian free group that has shadowing. This statement leads us to a question: which groups admit an action satisfying shadowing property?

These three results illustrate that the shadowing property depends not only on hyperbolic properties of actions of its elements but on the group structure as well.

The article is organized as follows. In Section 2 we give a definition of shadowing for actions of finitely generated groups. In Section 3 we recall necessary notions from group theory. In Section 4 we give precise statements of main results. Sections 5--7 are devoted to detailed consideration of actions of nilpotent, solvable Baumslag-Solitar, and free groups respectively.  In the appendix for consistency we prove Proposition~\ref{prop} about independence of shadowing on a choice of a generating set.

% Section~5 is devoted %to actions of nilpotent groups. In Section 6 we study the case of a solvable %Baumslag-Solitar group. In Section~7 we study actions of free groups.

\section{Main definitions.}

Let $G$ be a finitely generated (not necessarily abelian) group.
Let $\Omega$ be a metric space with a metric $\dist$. For $x \in \Omega$, $U \subset \Omega$, $\delta>0$ denote
$$
B(\delta, x) = \{y\in \Omega: \dist(x, y) < \delta\}, \quad
B(\delta, U) = \cup_{x \in U} B(\delta, x).
$$

We say that a map $\Phi: G \times \Omega \to \Omega$ is a (left) action of a group $G$ if the following holds:
\begin{itemize}
\item[(G1)] the map $f_g = \Phi(g, \cdot)$ is a homeomorphism of $\Omega$ for any $g \in G$;

\item[(G2)] $\Phi(e, x) \equiv x$, where $e \in G$ is the identity element of the group $G$;

\item[(G3)] $\Phi(g_1g_2, x) = \Phi(g_1, \Phi(g_2, x))$, for any $g_1, g_2 \in G$, $x \in \Omega$.
\end{itemize}

We say that an action $\Phi$ is \textit{uniformly continuous} if for some symmetric generating set~$S$ (a generating set is called symmetric if together with any element $s \in S$ it contains~$s^{-1}$) of a group $G$ the maps $f_{s}$ are uniformly continuous for all $s\in S$. Note that if $\Omega$ is compact, then any action of a finitely generated group is uniformly continuous.

Let us fix some finite symmetric generating set $S$ of a group $G$.

\begin{definition}\label{defPst}
For $d >0$ we say that a sequence $\{y_{g}\}_{g\in G}$ is a $d$-pseudo\-tra\-jec\-to\-ry of an action $\Phi$ (with respect to the generating set $S$) if
\begin{equation}
\label{dpstdef}
\dist(y_{sg}, f_s(y_g)) < d, \quad s \in S, g \in G.
\end{equation}
\end{definition}

\begin{definition}\label{defStSh}
We say that an uniformly continuous action $\Phi$ has the \textit{shadowing property} on a set $V \subset \Omega$ if for any $\ep > 0$ there exists $d > 0$ such that for any $d$-pseudotrajectory $\{y_g \in V\}_{g\in G}$ there exists a point $x_e \in \Omega$ such that
\begin{equation}\label{eqDefSh}
\dist(y_g, f_g(x_e)) < \ep, \quad g \in G.
\end{equation}
In this case we say that $\{y_g\}_{g\in G}$ is $\ep$-shadowed by the exact trajectory $\{x_g = f_g(x_e)\}_{g\in G}$.
If $V = \Omega$, we simply say that $\Phi$ has the shadowing property.
\end{definition}

This notion is a natural generalization of the concept of the shadowing property introduced in \cite{PilTikh} for actions of $\ZZ^n$.

Let us also note that the definition of a pseudotrajectory depends on a choice of the generating set $S$. However the following proposition shows that if an  uniformly continuous action has shadowing for one finite symmetric generating set, then it has shadowing for any finite symmetric generating~set.

\begin{proposition}
\label{prop}
Let $S$ and $S^{\prime}$ be two finite symmetric generating sets for a group $G$.  An uniformly continuous action $\Phi$ has the shadowing property on a set $V \subset \Omega$ with respect to the generating set $S$, if and only if it has the shadowing property on a set $V \subset \Omega$ with respect to the generating set $S^{\prime}$.
\end{proposition}

The proof of Proposition \ref{prop} is straightforward, see Appendix.

The following notion of expansivity is important for our results:
\begin{definition}
An action $\Phi$ is \textit{expansive} (or has \textit{expansivity}) on a set $U \subset \Omega$ if there exists $\Delta > 0$ such that if
$$
\Phi(g, x_1),\Phi(g, x_2) \in U, \quad \dist(\Phi(g, x_1), \Phi(g, x_2)) < \Delta, \quad \forall g \in G
$$
for some $x_1, x_2 \in U$, then
 $x_1 = x_2$.
\end{definition}
Note that if $G_1 \leq G$ is a subgroup of $G$ and $\Phi|_{G_1}$ has expansivity, then~$\Phi$ has expansivity too.

Any homeomorphism $f: \Omega \to \Omega$ induces an action $\Phi_f: \mathbb{Z} \times \Omega \to \Omega$ of the group $\mathbb{Z}$ defined as $\Phi_f(k, x) = f^k(x)$ for $k\in\mathbb{Z}$, $x\in\Omega$.
 We say that
\begin{enumerate}
\item a %uniformly continuous
homeomorphism $f$ has shadowing on a set \hbox{$V \subset \Omega$};
\item a homeomorphism $f$ has expansivity on a set $U \subset \Omega$;
\end{enumerate}
if the corresponding action $\Phi_f$ has this properties. Note that these definitions are equivalent to classical definitions of these notions.

\begin{definition}
Consider two sets $U, V \subset \Omega$. We say that an uniformly continuous action $\Phi$ is topologically Anosov with respect to the pair $(U, V)$ if the following conditions are satisfied:
\begin{itemize}
\item[(TA1)] there exists $\gamma > 0$ such that $B(\gamma, V) \subset U$;
\item[(TA2)] $\Phi$ has the shadowing property on $V$;
\item[(TA3)] $\Phi$ is expansive on $U$.
\end{itemize}
\end{definition}

\section{Finitely generated groups.}

In this section we will outline basic notions from theory of finitely generated groups,
give relevant definitions, and formulate statements that we use in the
sequel. We refer the interested reader to the following books on group theory:~\cite{Bech, BrHaef, PdH, Kur}.

A group $G$ is called \textit{abelian} if $[g,h]:=ghg^{-1}h^{-1}=e$ for any $g,h\in G$.

\begin{definition}
Any abelian group is called a \textit{nilpotent group of class 1}.
A group $G$ is called \textit{nilpotent of class $n$} if it has the \textit{lower central series} of length~$n$:
\begin{equation}\notag
\label{nilpdef}
G=G_1\rhd \ldots\rhd G_{n+1}=e,\qquad\mbox{where }G_{i+1}=[G_i,G],\quad G_n\neq e
\end{equation}
(as usual, $G_i\rhd G_{i+1}$ means that $G_{i+1}$ is a normal subgroup of $G_i$).
\end{definition}

The simplest nontrivial example of a nilpotent group is a so-called \textit{Heisenberg group} (see \cite{PdH}, \cite{Kur}): $<a,b,c\mid c=[a,b], ac=ca,bc=cb>$.

\begin{definition}
A group is called \textit{virtually nilpotent} if it has a nilpotent normal subgroup of a finite index (i.e. the corresponding factor group is finite).
\end{definition}

\begin{remark}\label{remFinGen}
Note that any subgroup of a finitely generated virtually nilpotent group is finitely generated. In fact the similar statement holds for a more general class of polycyclic groups (see \cite{BrHaef, Seg} for the details).
\end{remark}

Virtually nilpotent groups are important due to the celebrated theorem of Gromov: Any group of polynomial growth is virtually nilpotent. We refer the reader to \cite{Gromov} for the precise statement.

\begin{definition}
A group is called \textit{solvable} or \textit{soluble} if there exists a \textit{subnormal series} (of not necessarily finitely generated groups) $e=G_n\lhd\ldots\lhd G_1\lhd G_0=G$ such that $G_i/G_{i+1}$ is an abelian group. %This series are called \textit{subnormal} series.
\end{definition}

We study Baumslag-Solitar groups (see \cite{PdH}):
$$
BS(m,n)=<a,b\mid ba^m=a^nb>,\quad m,n\in\mathbb{Z},
$$
which are solvable for $m=1$. These groups are well known in group theory as a source of numerous counterexamples.

We study actions of $F_n=<a_1,\ldots,a_n\mid\cdot>$, the \textit{free group} with $n$ generators, which is obviously not solvable.

\section{Main results.}
%%%%%%%%%%%%%%%%%%%%%%%%%%%%%%%%%%%%%%%%%%%%%%%

The following theorem is the main result of our manuscript:

\begin{theorem}\label{thmVirtNilp}
Let $\Phi$ be an uniformly continuous action of a finitely generated virtually nilpotent group $G$ on a metric space $\Omega$. Assume that there exists an element $g \in G$ such that $f_g$ is topologically Anosov with respect to a pair $(U, V)$.
Then the action $\Phi$ is topologically Anosov with respect to the pair~$(U, V)$.
\end{theorem}

\begin{remark} This result generalizes \cite{PilTikh}, where a similar statement was proved for abelian~$G$.
\end{remark}

The main step of the proof is the following lemma, which is interesting by itself:
\begin{lemma}\label{lemNormSubgroup}
Let $G$ be a finitely generated group and $H$ be its finitely generated normal subgroup. Let $\Phi$ be an uniformly continuous action on $\Omega$. If $\Phi|_H$ is topologically Anosov with respect to a pair $(U, V)$, then $\Phi$ is topologically Anosov with respect to the pair $(U, V)$ too.
\end{lemma}

It turns out that Theorem \ref{thmVirtNilp} cannot be generalized to the case of solvable groups. Consider a solvable group $G = BS(1, n) = \; <a, b| ba = a^nb>$, where $n > 1$. For $\lam >0$ consider the action $\Phi: G \times \RR^2 \to \RR^2$ generated by the maps
$$
f_a(x) = Ax, \quad f_b(x) = B x,
$$
where
$$
A = \left(
      \begin{array}{cc}
        1 & 0 \\
        1 & 1 \\
      \end{array}
    \right), \quad B = \left(
                         \begin{array}{cc}
                           \lam & 0 \\
                           0 & n\lam \\
                         \end{array}
                       \right).
$$
Note that $BA = A^nB$, and hence the action $\Phi$ is well defined.

For any $\lam > 1$ the map $f_b$ is hyperbolic, however the following holds:
%\begin{remark}
%This action is quite degenerate (not faithful). However it is not difficult %to show that for any linear one-dimensional action of group $BS(1, n)$ holds %the relation $f_a^{n-1} = \mbox{Id}$.
%\end{remark}
\begin{theorem}
\label{ThBS}
\begin{itemize}
  \item[(i)] For $\lambda \in (1, n]$ the action $\Phi$ does not have the shadowing property.
  \item[(ii)] For $\lambda > n$ the action $\Phi$ has the shadowing property.
\end{itemize}
\end{theorem}

For actions of free groups we prove the following theorem:

\begin{theorem}
\label{thIMElin}
Any linear action of a
finitely generated free group with at least two generators on an Euclidian space does not have the shadowing property.
\end{theorem}

This theorem leads us to the following conjecture and question:

\begin{conjecture}
Any uniformly continuous action of a finitely generated free group with at least two generators on a manifold does not have the shadowing property.
\end{conjecture}
\begin{question}
Which groups admit an action on a manifold satisfying the shadowing property?
\end{question}

We derive Theorem \ref{thIMElin} from the following more general, but more technical statement:

\begin{theorem}
\label{thIME}
Let $G$ be a finitely generated free group with at least two generators. Let $\Phi$ be an uniformly continuous action of $G$ on a non-discrete metric space $\Omega$.
\begin{enumerate}
	\item If for some $g\in G$ the map $f_{g}$ is expansive, then $\Phi$ does not have shadowing.
  \item If for some $g\in G$, $g \ne e$, the map $f_{g}$ does not have shadowing, then $\Phi$ does not have shadowing too.
\end{enumerate}
\end{theorem}

\begin{remark}
\label{remIME}
Item 1 of Theorem \ref{thIME} holds for a more general class of groups with infinitely many ends (we refer the reader to \cite{BrHaef} for the precise definition).
\end{remark}

\begin{proof}[Proof of Theorem \ref{thIMElin}.]
Since for linear actions of $\mathbb{Z}$ both shadowing and expansivity are equivalent to hyperbolicity, Theorem \ref{thIMElin} follows from Theorem~\ref{thIME}.
\end{proof}

\section{Actions of nilpotent groups.}
We start from the proof of Lemma \ref{lemNormSubgroup}.

\begin{proof}[Proof of Lemma \ref{lemNormSubgroup}]
Fix a finite symmetric generating set $S_H$ in $H$ and continue it to a finite symmetric generating set $S$ in $G$. By Proposition \ref{prop}, we can assume that our initial generating set $S$ was chosen in this way.

Let $\Delta, \gamma > 0$ be the constants from the definitions of a topologically Anosov action and expansivity for $\Phi|_H$.
Since the maps $\{f_s\}_{s \in S}$ are uniformly continuous, there exists $\delta < \min(\Delta/3, \gamma)$ such that
\begin{equation}\label{eqUnCont}
\dist(f_s(\omega_1), f_s(\omega_2)) < \Delta/3
\end{equation}
for any $s \in S$ and any two points $\omega_1, \omega_2 \in \Omega$ satisfying $\dist(\omega_1, \omega_2) < \delta$.

Fix $\ep \in (0, \delta)$ and choose $d < \ep$ from the definition of shadowing for $\Phi|_H$ for the generating set $S_H$. Fix a $d$-pseudotrajectory $\{y_g \in V\}_{g\in G}$ of~$\Phi$.

For any element $q\in G$ consider the sequence $\{z_h = y_{hq}\}_{h\in H}$. Note that this sequence is a $d$-pseudotrajectory of $\Phi|_{H}$. Since $\Phi|_{H}$ is topologically Anosov with respect to $(U, V)$, there exists a unique point $x_q\in U$ such that
\begin{equation}
\label{shad1}
\dist(z_h,\Phi(h,x_q))=\dist(y_{hq},f_{h}(x_q)) < \ep, \quad h \in H.
\end{equation}
Existence of such $x_q$ follows from (TA2), uniqueness follows from (TA1), (TA3), and the inequality $\ep < \gamma$.

Let us prove that $\{x_q\}_{q\in G}$ is an exact trajectory.

Fix $s \in S$ and $q \in G$. Consider an arbitrary element $h\in H$. Since $H$ is a normal subgroup of $G$, there exists an element $h'\in H$ such that
\begin{equation}
\label{connect}
sh'=hs.
\end{equation}

It follows from \sref{eqUnCont}--\sref{connect} that
\begin{equation}
\label{ineq1}
\mbox{dist}(y_{sh'q},f_{h}(x_{sq})) < \epsilon,
\end{equation}
\begin{equation}
\label{shad2}
\mbox{dist}(f_s(y_{h'q}),f_s(f_{h'}(x_q))) < \Delta/3.
\end{equation}

Since $\{y_g\}_{g\in G}$ is a $d$-pseudotrajectory for $\Phi$, it follows from \sref{connect}--\sref{shad2} that
\begin{multline*}
\dist(f_{h}(x_{sq}),f_{h}(f_s(x_q)))\leq \\ \dist(f_{h}(x_{sq}), y_{sh'q}) + \dist(y_{sh'q}, f_{s}(y_{h'q}))+
\dist(f_{s}(y_{h'q}), f_{hs}(x_q)))\leq \\ \epsilon + d + \Delta/3 < \Delta.
\end{multline*}

Due to expansivity of $\Phi|_{H}$ on $U$, we conclude that
\begin{equation}\notag
\label{goal}
x_{sq} = f_s(x_q), \quad s \in S, q \in G.
\end{equation}
Since $S$ is a generating set for $G$, these equalities imply that
$x_q = f_q(x_e)$, for  all $q \in G$, and hence by~\sref{shad1}  $x_e$ satisfies inequalities \sref{eqDefSh}.

Expansivity of $\Phi$ is trivial.
\end{proof}
%\end{equation}

%Let us prove that
%\begin{equation}\label{eqTraj}
%_q = f_q(x_e), \quad q \in G.
%\end{equation}
%Note that in that case inequalities \sref{shad1} imply that $x_e$ is desired %element for the shadowing property. Clearly in order to prove \sref{eqTraj} %it is enough to prove
%\begin{equation}
%label{goal}
%y_{sq} = f_s(y_q), \quad s \in S, q \in G.
%\end{equation}

%%%%%%%%%%%%%%%%%%%%%%%%%%%%%%%%%%%%%%%

Next we prove Theorem \ref{thmVirtNilp} for the case of nilpotent groups.
\begin{lemma}\label{lemNilp}
Let $G$ be a finitely generated nilpotent group of class $n$ and $\Phi$ be an uniformly continuous action of $G$ on a metric space $\Omega$. Assume that there exists an element $g\nobreak \in\nobreak G$ such that $f_g$ is topologically Anosov with respect to $(U, V)$.
Then the action $\Phi$ is topologically Anosov with respect to $(U, V)$.
\end{lemma}
\begin{proof}
Let us prove this lemma by induction on $n$.

For $n = 1$ the group $G$ is abelian and hence the group $P = \left<g\right>$ generated by $g$ is a normal subgroup of $G$. Since $f_g$ is topologically Anosov, applying Lemma~\ref{lemNormSubgroup} we conclude that $\Phi$ has the shadowing property.

Let $n > 1$ and assume that we have proved the lemma for all nilpotent groups of class less or equal $n-1$. Denote $Q = [G, G]$ and $P = \left< Q, g \right>$ (i.e. $P$ is the minimal subgroup of $G$ that contains $Q$ and $g$).
\begin{proposition}
\label{nilpprop}
%Let $G$ be a nilpotent group of class $n$, put $Q=[G,G]$. Fix $g\in G$ and %put $P=<Q,g>$.
\begin{itemize}
  \item[(N1)] The group $P$ is a normal subgroup of $G$.
    \item[(N2)] The group $P$ is nilpotent of class at most $n-1$.
  %\item[(N3)] The group $P$ is finitely generated.
\end{itemize}
\end{proposition}

\begin{proof}[Proof of Proposition \ref{nilpprop}.]
Let us start from Item (N1). Fix arbitrary $p\in P$, $h\in G$. Note that $hph^{-1}\in [G,G]p=Qp\subset P$, which proves the claim.

Let us prove Item (N2).
It is clear that any subgroup of a nilpotent group of class $n$ is a nilpotent group of class at most $n$. However we need a stronger result for the subgroup~$P$. As the analysis of simple examples shows (e.g. the direct product of the Heisenberg group and $\mathbb{Z}$), a nilpotent group of class $n$ may have proper subgroups of class $n$. So item (N2) is not trivial.

Denote
\begin{equation}\notag
\label{qdef}
R = [Q,G]=[[G,G],G].
\end{equation}
Clearly, in order to prove (N2) it is sufficient to prove that
\begin{equation}\notag
\label{tempgoal}
[P,P]\subset [[G,G],G] = R
\end{equation}
(since it implies $[[P,P],P]\subset [[[G,G],G],G]$ and etc.).

Since $Qg=gQ$, any element $p\in P$ has a representation as $qg^k$ for some $q\in Q$, $k\in\mathbb{Z}$. Fix $p_1,p_2\in P$ and put $p_1=q_1g^{k_1}, p_2=q_2g^{k_2}$.
%By $(\ref{qdef})$ there exists $r_1\in R$ such that
Note that
$$
p_1p_2=q_1g^{k_1} q_2g^{k_2} = q_1r_1q_2g^{k_1+k_2} = r_2q_1q_2g^{k_1+k_2},
$$
\begin{equation}\notag
%\label{mform2}
p_2p_1 = q_2g^{k_2}q_1g^{k_1} = q_2r_3q_1g^{k_1 + k_2} = r_4q_2q_1g^{k_1+k_2} = r_5q_1q_2g^{k_1+k_2}
\end{equation}
for some $r_1, \dots, r_5 \in R$, and hence $[p_1,p_2] = r_2r_5^{-1}\in R.$

%Note that $q_1r_1q_1^{-1}r^{-1}\in [Q,G] = R$, hence there exists $r_2\in R$  %such that
%\begin{equation}
%\label{mform1}
%p_1p_2 = r_2q_1q_2g^{k_1+k_2}.
%\end{equation}
%In an analagous way there exist $r_3, r_4\in R$ such that
%It follows from $(\ref{mform1})$ and $(\ref{mform2})$ that
%$[p_1,p_2] = r_2r_4^{-1}\in R.$

%Note that Item (N3) holds not only for nilpotent groups, but also for a more general class of polycyclic groups (cf. Definition \ref{polycdef}), since any subgroup of a polycyclic group is finitely generated.

\end{proof}

Let us note that these properties strongly use nilpotency of $Q$, and their analogs do not hold, for example, for solvable groups.

Let us continue the proof of Lemma \ref{lemNilp}. Since $P$ is a finitely generated (due to Remark~\ref{remFinGen}) nilpotent group of class at most $n-1$, $g \in P$, and $f_g$ is topologically Anosov, by the induction assumptions we conclude that $\Phi|_P$ is topologically Anosov. Combining this property, (N1) and Lemma \ref{lemNormSubgroup} we conclude that $\Phi$ has the shadowing property.
\end{proof}

\begin{proof}[Proof of Theorem \ref{thmVirtNilp}]
Since $G$ is virtually nilpotent, there exists a nilpotent normal subgroup $H$ of $G$ of finite index. Due to Remark \ref{remFinGen} the group $H$ is finitely generated. Consider $g\in G$ from the assumptions of the theorem. Since $H$ is a subgroup of finite index, there exists $k > 0$ such that $g^k \in H$. Since $f_g$ is topologically Anosov, the map $f_g^k = f_{g^k}$ is also topologically Anosov. Hence, by Lemma \ref{lemNilp}, the action $\Phi|_{H}$ is topologically Anosov. Applying Lemma \ref{lemNormSubgroup} we conclude that $\Phi$ is topologically Anosov~too.
\end{proof}

\section{An action of a Baumslag-Solitar group.}

\begin{proof}[Proof of Theorem \ref{ThBS}]

Without loss of generality, by Proposition~\ref{prop}, we consider the group $BS(1,n)=\allowbreak<~a,b\mid ba=a^nb>$ with the standard generating set $S=\{a,b,a^{-1},b^{-1}\}$. Denote by $P_1$ and $P_2$ the natural projections on the coordinate axes in $\RR^2$. As before denote
$$
A=\left(
\begin{array}{cc} 1 & 0 \\ 1 & 1 \end{array}
\right),
\quad
B=\left(
\begin{array}{cc} \lambda & 0 \\ 0 & n\lambda \end{array}
\right).
$$
Note that
\begin{equation}
\label{matproducts}
A^r=\left(
\begin{array}{cc} 1 & 0 \\ r & 1 \end{array}
\right),
\quad
B^r=\left(
\begin{array}{cc} \lambda^r & 0 \\ 0 & (n\lambda)^r \end{array}
\right),\qquad  r\in \ZZ.
\end{equation}

\textbf{Proof of Item (i).} To derive a contradiction assume that $\Phi$ has shadowing and choose $d > 0$ from the definition of the shadowing property applied to $\epsilon=1$.

Consider an auxiliary action $\Psi:  G \times (\RR \times \ZZ) \to (\RR \times \ZZ)$ generated by the maps
$$
g_a(x, k) = (x + n^{-k}, k), \quad g_b(x, k) = (x, k+1).
$$
It is easy to check that $g_b \circ g_a = g_a^n \circ g_b$, and hence the action $\Psi$ is well defined.

Consider the map $F : (\RR \times \ZZ) \to \RR$ defined as follows
$$
F(x, k) = \left((1+ \beta)\lam^k |x|^{\beta}; (n\lam)^k |x|^{1+\beta}\right),
%  \begin{cases}
%  0, & \quad k <0, \quad \mbox{or} \quad x \leq 0,\\
%  x \lambda^k, & \quad x \in (0, 1), \; k \geq 0,\\
%  \lambda^k, & \quad x \geq 1, \; k \geq 0.\\
% \end{cases}.
$$
where $\beta = \frac{\ln \lam}{\ln n} \in (0, 1]$.

Finally, consider the sequence
$$
y_g = \frac{d}{3} \cdot F(\Psi(g, (0, 0))), \quad g \in G.
$$
We claim that $\{y_g\}_{g\in G}$ is a $d$-pseudotrajectory for the action $\Phi$, i.e. inequalities $(\ref{dpstdef})$ hold for $s\in \{a,b,a^{-1},b^{-1}\}$.

Indeed, fix $g \in G$. Denote $(x, k) = \Psi(g, (0, 0))$.

If $s = b^{\pm 1}$, then it is easy to see that $y_{sg} = f_s(y_g)$.

If $s = a$, then
$$
P_1 y_{sg} = \frac{d}{3}\left(1+ \beta\right)\lam^k \left|x + n^{-k}\right|^{\beta},\quad P_2 y_{sg} = \frac{d}{3}\left(n\lam\right)^k \left|x+ n^{-k}\right|^{1+\beta}.$$
%y_{sg} = \frac{d}{3}\left(\left(1+ \beta\right)\lam^k \left|x + n^{-k}\right|^{\beta}; \left(n\lam\right)^k \left|x+ n^{-k}\right|^{1+\beta}\right).

Denote $\Delta = n^{-k}$. Then $\lam^k = \Delta^{-\beta}$. In such notation
$$P_1\left(y_{ag} - f_a(y_g)\right) = \frac{d}{3}\left(1+ \beta\right)\Delta^{-\beta}\left(\left|x+\Delta\right|^{\beta} - \left|x\right|^{\beta}\right),$$
$$P_2\left(y_{ag} - f_a(y_g)\right) = \frac{d}{3}\left(\Delta^{-(1+ \beta)}\left|x+\Delta\right|^{1+\beta} - \left(\Delta^{-(1+ \beta)}|x|^{1+\beta} + \left(1+ \beta\right)\Delta^{-\beta} |x|^{\beta}\right)\right).$$
%\begin{multline*}
%y_{ag} - f_a(y_g) =\\
%=\frac{d}{3}\left(\left(1+ \beta\right)\Delta^{-\beta}\left(\left|x+\Delta\right|^{\beta} - \left|x\right|^{\beta}\right); \Delta^{-(1+ \beta)}\left|x+\Delta\right|^{1+\beta} - \left(\Delta^{-(1+ \beta)}|x|^{1+\beta} + \left(1+ \beta\right)\Delta^{-\beta} |x|^{\beta}\right)\right).
%\end{multline*}

We use the following inequalities, which hold for $\beta \in (0, 1]$ and $x, \Delta \in \RR$,:
$$
|x+ \Delta|^{\beta} \leq |x|^{\beta} + |\Delta|^{\beta},
$$
$$
|x+\Delta|^{1+\beta} \leq |x|^{1+\beta} + (1+\beta)|\Delta||x|^{\beta} + |\Delta|^{1+\beta}.
$$
From these inequalities it is easy to conclude that

$$
|P_1(y_{ag} - f_a(y_g))| \leq (\beta + 1)d/3, \quad  |P_2(y_{ag} - f_a(y_g))| \leq d/3
$$
which implies
$$
|y_{ag} - f_a(y_g)| < d.
$$
Similarly
$$
|y_{a^{-1}g} - f_{a^{-1}}(y_g)| < d.
$$
And hence $\{y_g\}_{g\in G}$ is a $d$ pseudotrajectory.

%\textit{Case $s = a$} (case $s=a^{-1}$ is treated similarly). Note that %$$\Psi(ag, (0, 0)) = (x + n^{-k}, k),$$
%$$
%|y_{ag} - f_a(y_g)| = |F(x + n^{-k}, k) - F(x, k)|\cdot d/2.
%$$

%If $k < 0$ or $x \in (-\infty, -n^{-k}] \cup (1, +\infty)$, then
%$
%F(x + n^{-k}, k) = F(x, k).
%$

%Otherwise $k \geq 0$ and one of the following holds:
%\begin{itemize}
%\item $x \in (-n^{-k}, 0]$ and $|F(x + n^{-k}, k) - F(x, k)| = %|(x+n^{-k})\lambda^k|\leq n^{-k}\lambda^k \leq 1;$
%\item $x \in (0, 1- n^{-k}]$ and
%$|F(x + n^{-k}, k) - F(x, k)| = n^{-k}\lambda^{k}\leq 1;$
%\item $x \in (1- n^{-k}, 1]$ and $|F(x + n^{-k}, k) - F(x, k)| = \lambda^k %(1-x)\leq \lambda^k n^{-k}\leq 1.$
%\end{itemize}

%Thus in all cases
%$|y_{ag} - f_a(y_g)| < d$, which proves (\ref{dpstdef}).

%\textit{Case $s = b$} (case of $s=b^{-1}$ is treated similarly). Note that
%$$\Psi(bg, (0, 0)) = (x, k+1),$$
%$$
%|y_{bg} - f_b(y_g)| = |F(x, k) - \lam F(x, k-1)|\cdot d/2.
%$$

%If $k \ne 0$ or $x \notin [0, 1]$, then $F(x, k) = \lam F(x, k-1)$.

%Otherwise $k = 0$, $0\leq x\leq 1$, and hence
%$|F(x,k) - \lambda F(x,k-1)| = |x|\leq1.$

%Thus $|y_{bg} - f_b(y_g)| < d$, which proves (\ref{dpstdef}).

Since by our assumptions the action $\Phi$ has the shadowing property, there exists $x_e \in \RR$ such that
$(\ref{eqDefSh})$ holds.

Note that $y_{b^k} = 0$ for $k \geq 0$. Substituting $g = b^k$ into (\ref{eqDefSh}), we conclude that
$
|B^kx_e|\leq 1
%\left| \left(
%  \begin{array}{cc}
%    \lam^k & 0 \\
%    0 & (n\lam)^k \\
%  \end{array}
%\right)x_e \right| \leq 1
$
and hence, by expansivity of $f_b$, $x_e = (0, 0)$.

Now substituting $g = b^k a$ into $(\ref{eqDefSh})$ and looking on the first coordinate we conclude that:
$$
|\lambda^kd/2 - 0|<1, \quad k  \geq 0,
$$
which is impossible for sufficiently large $k$.
The derived contradiction finishes the proof of Item (i).

\medskip

\textbf{Proof of Item (ii).} Fix $\ep > 0$. Note that the map $f_b$ has shadowing and expansivity. Let us choose $d \in (0, \ep)$ such that any $d$-pseudotrajectory of $f_b$ can be $\ep$-shadowed by an exact trajectory of $f_b$. Consider an arbitrary $d$-pseudotrajectory $\{y_g\}_{g \in G}$ of the action~$\Phi$.

For any element $q\in G$ consider the sequence $\{z_k\}_{k\in \ZZ}$, defined by $z_k= y_{b^k q}$. Note that this sequence is a $d$-pseudotrajectory for $f_b$. Since $f_b$ has shadowing and expansivity, there exists a unique point $x_q\in \RR$ such that
\begin{equation}
\label{BSshad1}
|z_k - f_b^k(x_q)|=|y_{b^k q}-f_{b}^k(x_q)| < \ep, \quad k \in \ZZ.
\end{equation}
We claim that $x_q = \Phi(q, x_e)$. To prove this, it is enough to show that
\begin{equation}\label{eqBSshad2}
x_{bq} =% \left(
%\begin{matrix}
%\lambda & 0 \\ 0 & n\lambda
%\end{matrix}
%\right)
%x_q,
Bq,
%\;
\quad x_{aq} =
%\left(
%\begin{matrix}
%1 & 0\\
%1 & 1
%\end{matrix}
%\right)
A x_q, \qquad q \in G.
\end{equation}
The first equality follows directly from expansivity of $f_b$. Let us prove the second one.

Note that the relation $ba = a^n b$ implies that \begin{equation}\label{eqGroupRel}
b^ka = a^{(n^k)}b^k, \quad k > 0.
\end{equation}
Fix an arbitrary $q \in G$.  Note that since $\{y_t\}_{t\in G}$ is a $d$-pseudotrajectory,
\begin{equation}
\label{proj1}
|P_1 y_{a^{n^k}b^kq} - P_1 y_{b^kq}| < dn^k.
\end{equation}
By a straightforward induction it is easy to show that for $j \in [1, n^k]$ the inequality
%$$|pr_2 y_{a^{n^k}b^kq} - pr_2y_{(a^{n^k - 1})b^kq} - pr_1y_{(a^{n^k - %1})b^kq}|< d_1:=d,$$
%$$|pr_2 y_{a^{n^k}b^kq} - pr_2y_{(a^{n^k-2})b^kq} - 2pr_1y_{(a^{n^k - %2})b^kq}|< d_2:=d + 2d=3d,$$
$$
|P_2 y_{a^{n^k}b^kq} - P_2y_{(a^{n^k-j})b^kq} - jP_1y_{(a^{n^k - j})b^kq}|< \frac{j(j+1)}{2}d
$$
holds.
In particular
\begin{equation}
\label{proj2}
|P_2 y_{a^{n^k}b^kq} - P_2 y_{b^kq} - n^kP_1y_{b^kq}| < \frac{n^k(n^k+1)}{2}d.
\end{equation}

Relations (\ref{BSshad1}), (\ref{eqGroupRel}), and the definition of a pseudotrajectory imply that for any $k > 0$ the following relations hold:
$$
%|\lam^k pr_1 x_{aq} - pr_1 y_{b^kaq}| < \ep,\quad |(n\lam)^k pr_2 x_{aq} - pr_2 y^{b^kaq}|
%\left|\left(
%\begin{matrix}
%\lambda^k & 0 \\ 0 & (n\lambda)^k
%\end{matrix}
%\right)
\left|B^kx_{aq} - y_{b^kaq}\right| < \ep,
$$
$$
y_{b^kaq} = y_{a^{(n^k)}b^kq},
$$
%$$
%|y_{a^{(n^k)}b^kq} - y_{b^kq}| < d n^k,
%$$
%$$
%|y_{b^kq} - \lam^k x_q| < \ep,
%$$
%|\lam^k pr_1 x_{aq} - pr_1 y_{b^kaq}| < \ep,\quad |(n\lam)^k pr_2 x_{aq} - pr_2 y^{b^kaq}|
$$
\left|y_{b^kq} - %\left(
%\begin{matrix}
%\lambda^k & 0 \\ 0 & (n\lambda)^k
%\end{matrix}
%\right)
B^k
x_q\right| < \ep,
$$
and hence by $\sref{matproducts}$, $\eqref{proj1}$ and $\eqref{proj2}$
%$$
%|\lam^k(x_{aq} - x_q)| < 2\ep + d n^k, \quad\forall k > 0.
%$$
$$
|\lam^k (P_1x_{aq} - P_1x_q)|<2\ep + dn^k,\quad k>0,
$$
$$|(n\lam)^k P_2 x_{aq} - (n\lam)^k P_2 x_{q} - (n\lam)^k P_1 x_q|<2\ep + \frac{n^k(n^k+1)}{2}d, \quad k>0.$$
Since $\lam > n$,
$$P_1 x_{aq} = P_1 x_q,\quad P_2 x_{aq} = P_2 x_{q} + P_1 x_q,$$
which implies~\sref{eqBSshad2} and finishes the proof of Item (ii).

\end{proof}

%\section{Absense of shadowing for actions of groups with infinitely many ends.}
\section{Actions of free groups.}

Without loss of generality, by Proposition 1, we consider a free group $G=\allowbreak=<~a_1,\ldots,a_n|\cdot>$ with the standard generating set $S = \{a_1^{\pm 1},\ldots,a_n^{\pm 1}\}$.
It means that any element $g\in G$ has a normal form $g=s_r\ldots s_1$ (where $s_j\in S$), i.e. the unique shortest representation in terms of elements of $S$.

\begin{proof}[Proof of Theorem \ref{thIME}.]

%We will start from \textbf{Item 1}.
\textbf{Proof of Item 1}.
 To derive a contradiction, suppose that $\Phi$ has shadowing. Let $d$ be the number that corresponds to $\epsilon=\Delta$ (the constant of expansivity of~$f_g$) in the definition of shadowing for $\Phi$.

Consider the normal form of $g$: $g=s_r\ldots s_1$.
Fix any $q\in S\backslash \{s_1, s_1^{-1}\}$.
%Let $d$ be a small number.
Since $f_{q}^{-1}$ is uniformly continuous, there exists a number $d_1<d$ such that
\begin{equation}
\label{unifcont}
\mbox{dist}(f_{q}^{-1}(w_1),f_{q}^{-1}(w_2))< d,
\end{equation}
for any $w_1,w_2\in\Omega$ satisfying $\mbox{dist}(w_1,w_2)< d_1$.

Since $\Omega$ is non-discrete, we can fix two distinct points $\omega_0,\omega\in \Omega$ such that $\mbox{dist}(\omega_0,\omega)< d_1$. We construct a pseudotrajectory $\{y_t\}_{t\in G}$ in the following way:
$$
y_t = 
\begin{cases}
\Phi(t,f_{q}^{-1}(\omega)),& \mbox{if the normal form of $t\in G$ starts with $q$},\\
\Phi(t,f_{q}^{-1}(\omega_0)),& \mbox{otherwise}.
\end{cases}
$$
Note that, by $(\ref{unifcont})$,
$$\mbox{dist}(y_{q},f_{q}(y_e)) = \mbox{dist}(\omega,\omega_0)< d_1< d,$$
$$
\mbox{dist}(y_e, f_{q}^{-1}(y_{q})) = \mbox{dist}(f_{q}^{-1}(\omega_0),f_{q}^{-1}(\omega))< d,
$$
and the equality $y_{st} = f_s(y_t)$ holds for all other $s\in S$, $t\in G$. Hence $\{y_t\}_{t\in G}$ is a $d$-pseudotrajectory.

Our assumptions imply the existence of a point $x_e$ such that inequalities~$(\ref{eqDefSh})$ hold.
Consequently,
$$\mbox{dist}(y_{g^k},\Phi(g^k,x_e))=\mbox{dist}(f_g^k(f_{q}^{-1}(\omega_0)) ,f_g^k(x_e))< \Delta,\quad\forall k\in\mathbb{Z},$$
which, by expansivity, implies that
\begin{equation}
\label{contr1}
x_e = f_q^{-1}(\omega_0).
\end{equation}
Since the normal form of $\{g^kq\}_{k\in\mathbb{Z}}$ starts from $q$,
$$\mbox{dist}(y_{g^kq},\Phi(g^kq,x_e))=\mbox{dist}(f_{g}^k(\omega), f_{g}^k(f_q(x_e)))< \Delta,\quad\forall k\in\mathbb{Z}.$$
Hence, by expansivity, $\omega = f_q(x_e)$, which together with $(\ref{contr1})$ contradicts to the choice of $\omega$ and $\omega_0$.
Thus $\Phi$ does not have shadowing, which proves Item~1.

\textbf{Proof of Item 2.}
Let $\epsilon$ be any number such that for any $d<\epsilon$ the map $f_g$ has a $d$-pseudotrajectory that cannot be $\epsilon$-shadowed by any exact trajectory of $f_g$.
Consider the normal form of $g=s_r\ldots s_1$.
Fix any $d<\epsilon$. There exists a number $d_1<d$ such that for any $\phi$ that has a form $\phi=f_{s_j}\ldots f_{s_1}$ or $\phi= f_{s_j^{-1}} \ldots f_{s_{r}^{-1}}$ for some $1\leq j\leq r$ we have
\begin{equation}
\label{importref}
\mbox{dist}(\phi(w_1),\phi(w_2))\leq d,
\end{equation}
for all $w_1,w_2\in\Omega$ such that $\mbox{dist}(w_1,w_2)\leq d_1$.

%By $\eqref{importref}$, there exists a sequence $\{z_k\}_{k\in\mathbb{Z}}$ such that
%$$z_{rk} = x_k,\quad k\in\mathbb{Z},$$
%$$\mbox{dist}(z_{rk+j}, f_{s_j}(z_{rk+j-1}))\leq d,\qquad 1\leq j\leq r,\quad k\geq0,$$
%$$\mbox{dist}(z_{rk+j-1}, f_{s^{-1}_{j}}(z_{rk+j}))\leq d,\qquad 1\leq j\leq r,\quad k\geq0,$$
%$$\mbox{dist}(z_{rk-j}, f_{s^{-1}_{r-j+1}}(z_{rk-j+1}))\leq d,\qquad 1\leq j\leq r,\quad k\leq0,$$
%$$\mbox{dist}(z_{rk-j-1}, f_{s_{r-j+1}}(z_{rk-j}))\leq d,\qquad 1\leq j\leq r,\quad k\leq0.$$

Consider a $d_1$-pseudotrajectory $\{x_k\}_{k\in\mathbb{Z}}$ for $f_g$ that cannot be $\epsilon$-shadowed and
the sequences $\{z_k\}_{k\in\mathbb{Z}}$, $\{y_t\}_{t \in G}$ defined as follows
$$
\begin{cases}
z_{rk} = x_k, & \quad k\in\mathbb{Z},\\
z_{rk+j+1} = f_{s_{j+1}}(z_{rk+j}), & \quad 0 \leq j < r-1, \quad k \in \mathbb{Z};\\
\end{cases}
$$
and
$$
y_t = \begin{cases}
\Phi(v,z_{rk+j}), & \mbox{for $w=tv^{-1}=s_j\ldots s_1(s_r\ldots s_1)^k$, $k\geq0$, $1\leq j\leq r$};\\
\Phi(v,z_{-rk-j}), & \mbox{for $w=tv^{-1}=s_{r-j+1}^{-1}\ldots s_r^{-1}(s_r\ldots s_1)^{-k}$, $k\geq0$, $1\leq j\leq r$};%\\
%\Phi(t,y_{}), & \mbox{otherwise},
\end{cases}$$
where $v$ is the element of minimal length such that $t = vw$ for some $w=tv^{-1}$ of the form defined above.

%\begin{itemize}
%\item $y_t=z_{rk+j}$ for $t=s_j\ldots s_1(s_r\ldots s_1)^k$, $k\geq0$, $1\leq %j\leq r$,
%\item $y_t = z_{-rk-j}$ for $t=s_j^{-1}\ldots s_r^{-1}(s_r\ldots s_1)^{-k}$, %$k\geq0$, $1\leq j\leq r$,
%	\item $y_t = \Phi(t,x_0)$ otherwise.
%\end{itemize}

%\begin{equation}
%\label{seteq1}
%\mbox{dist}(z_{rk+j}, f_{s_j}(z_{rk+j-1}))\leq d,\qquad 1\leq j\leq r,\quad %k\geq0,
%\end{equation}
%\begin{equation}
%\mbox{dist}(z_{rk+j-1}, f_{s^{-1}_{j}}(z_{rk+j}))\leq d,\qquad 1\leq j\leq %r,\quad k\geq0,
%\end{equation}
%\begin{equation}
%\mbox{dist}(z_{rk-j}, f_{s^{-1}_{r-j+1}}(z_{rk-j+1}))\leq d,\qquad 1\leq %j\leq r,\quad k\leq0,
%\end{equation}
%\begin{equation}
%\label{seteq4}
%\mbox{dist}(z_{rk-j-1}, f_{s_{r-j+1}}(z_{rk-j}))\leq d,\qquad 1\leq j\leq %r,\quad k\leq0.
%\end{equation}

%We construct a sequence $\{y_t\}_{t\in G}$ in the following way:

By $\eqref{importref}$ the sequence $\{y_t\}_{t\in G}$ is a $d$-pseudotrajectory. If it is $\ep$-shadowed by the trajectory of a point $u_e$, then $\{x_k\}_{k\in\mathbb{Z}}$ is $\ep$-shadowed by $\{f_g^k(u_e)\}_{k\in\mathbb{Z}}$, which leads to a contradiction.
\end{proof}

%By the construction of $\{z_k\}_{k\in\mathbb{Z}}$ and $\{y_t\}_{t\in G}$, the %existence of a point $x_e$ such that the analog of $(\ref{eqDefSh})$ holds %would imply that $\{x_k\}_{k\in\mathbb{Z}}$ can be $\ep$-shadowed by some %exact trajectory of $f_g$, which is not true. Item 2 is proved.

%\end{proof}

\section{Appendix.}

\begin{proof}[Proof of Proposition \ref{prop}]
The generating set $S$ induces on $G$ a so-called \textit{word norm} defined as length of the shortest representation of an element in terms of elements from $S$. We define by $|g|_S$ (or simply by $|g|$) the word norm of an element $g$ with respect to $S$.

It is well known that two word norms corresponding to two finite generating sets $S$ and $S^{\prime}$ are bilipschitz equivalent, i.e. there exists a constant $C\geq 1$ such that
\begin{equation}
\label{normequiv}
|g|_{S^\prime}/C\leq |g|_S\leq C|g|_{S^\prime},\quad\forall g\in G.
\end{equation}

Fix an $\epsilon>0$. Let $d$ be the number from the definition of shadowing of~$\Phi$ with respect to $S$ corresponding to $\epsilon$.
By uniform continuity of $\Phi$, there exists a constant $d_1< d/C$ such that %for any $|g|_{S^{\prime}}\leq C$ and any two points $\omega_1,\omega_2\in \Omega$ such that $\mbox{dist}(\omega_1,\omega_2)< d_1$ we have
\begin{equation}
\label{uniformcont}
\mbox{dist}(f_g(\omega_1),f_g(\omega_2))< d/C,
\end{equation}
for any $g\in G$, $\omega_1,\omega_2\in \Omega$ provided $|g|_{S^{\prime}}\leq C$ and $\mbox{dist}(\omega_1,\omega_2)< d_1$.

Let $\{y_{g} \in V\}_{g\in G}$ be a $d_1$-pseudotrajectory of $\Phi$ with respect to the generating set $S^{\prime}$, i.e. by $(\ref{dpstdef})$
$$
\mbox{dist}(y_{s^{\prime}g}, f_{s^{\prime}}(y_{g}))< d_1,\quad s^{\prime}\in S^{\prime}, g\in G
$$
%Consequently, by equicontinuity of $\{f_{g}\}|_{|g|_{S^{\prime}}\leq C}$,
Consequently, by $(\ref{uniformcont})$,

\begin{multline*}
\mbox{dist}(y_{s^{\prime}_C\ldots s^{\prime}_1 g},f_{s^{\prime}_C\ldots s^{\prime}_1} (y_g))\leq
\mbox{dist}(y_{s^{\prime}_C\ldots s^{\prime}_1 g}, f_{s^{\prime}_C}(y_{s^{\prime}_{C-1}\ldots s^{\prime}_1 g}))+ \\
+ \mbox{dist}(f_{s^{\prime}_C}(y_{s^{\prime}_{C-1}\ldots s^{\prime}_1 g}), f_{s^{\prime}_C}(f_{s^{\prime}_{C-1}}(y_{s^{\prime}_{C-2}\ldots s^{\prime}_1 g}))) + \ldots + \\
+\mbox{dist}( f_{ s^{\prime}_C\ldots s^{\prime}_2 }( y_{s^{\prime}_1 g} ), f_{s^{\prime}_C\ldots s^{\prime}_2}(f_{s^{\prime}_1}(y_g)))< d_1 + d/C + \ldots + d/C<d
\end{multline*}
for any $s^{\prime}_1,\ldots, s^{\prime}_C\in S^{\prime}$, $g\in G$.
To summarize,
\begin{equation}
\label{dpstdefmod}
\mbox{dist}(y_{hg}, f_{h}(y_{g}))< d,
\end{equation}
for all $g\in G$ and $h \in G$ such that $|h|_{S^{\prime}}\leq C$ .
It follows from $(\ref{normequiv})$ that any element $s\in S$ satisfies $|s|_{S^{\prime}}\leq C$.
Thus it follows from $(\ref{dpstdefmod})$ that the sequence $\{y_{g}\}_{g\in G}$ is a $d$-pseudotrajectory of $\Phi$ with respect to the generating set $S$.

It follows from our assumptions that $\{y_{g}\}_{g\in G}$ is $\epsilon$-shadowed by some point $x_e$. However inequalities $(\ref{eqDefSh})$ do not depend on the choice of the generating set. Thus $\Phi$ has shadowing with respect to $S^{\prime}$. Clearly $\Phi$ is an uniformly continuous action with respect to $S^{\prime}$ too.

\end{proof}

\section{Acknowledgment.}

Both authors are grateful to the organizers of Montevideo Dynamical Systems Conference 2012 for hospitality and creation of a fruitful atmosphere for the discussions that eventually led to the creation of this paper.
The work of both authors was partially supported  by the Chebyshev Laboratory (Department of Mathematics and Mechanics, St. Petersburg State University)  under RF Government grant 11.G34.31.0026 and JSC ``Gazprom Neft''. The work of the first author was partially supported by and performed at Centro di ricerca matematica Ennio De Giorgi (Italy). Also the work of the first author was partially supported by the travel grant 12-01-09302-mob\_z of RFFI. Work of the second author was also partially supported by the Humboldt Foundation (Germany).

\end{document}